\documentclass{amsproc}
\usepackage{amssymb}
\newtheorem{theorem}{Theorem}[section]
\newtheorem{lemma}[theorem]{Lemma}
\newtheorem{corollary}[theorem]{Corollary}

\theoremstyle{definition}
\newtheorem{definition}[theorem]{Definition}
\newtheorem{example}[theorem]{Example}

\usepackage{float}
\usepackage{tikz}
\usetikzlibrary{graphs, graphs.standard}
\usetikzlibrary{decorations.markings}

\usepackage{color}

   \usepackage[normalem]{ulem}

\theoremstyle{remark}
\newtheorem{remark}[theorem]{Remark}

\numberwithin{equation}{section}

\begin{document}

\title{Inverse $Z$-matrices with the bi-diagonal south-west structure}

\author{S. Pratihar}
\address{Department of Mathematics, Indian Institute of 
		Technology Madras,
		Chennai, 600036, India }
% \curraddr{Department of Mathematics, Indian Institute of 
% 		Technology Madras,
% 		Chennai, 600036, India }
\email{ma22d019@smail.iitm.ac.in (Samapti Pratihar)}
\thanks{The authors thank the referee for comments that have improved the readability. The first author was supported by the Prime Minister's Research Fellowship (PMRF), Ministry of Education, Government of India}

\author{K.C. Sivakumar}
\address{Department of Mathematics, Indian Institute of 
		Technology Madras,
		Chennai, 600036, India }
\email{kcskumar@iitm.ac.in (K.C. Sivakumar)}

%    General info
\date{}

 % \dedicatory{This paper is dedicated to our advisors.}


\begin{abstract}
Two new matrix classes are introduced; inverse cyclic matrices and bi-diagonal south-west matrices. An interesting relation is established between these classes. Applications to two classes of inverse $Z$-matrices are provided.
\end{abstract}

\maketitle

\section{Introduction}

\subsection{Inverse $Z$-matrices: a brief survey}

The objective of this article is twofold. One is to present a brief survey on the class of inverse $Z$-matrices, with specific reference to its nonnegativity/nonpositivity properties. The second aim is to propose two matrix classes, motivated by a result on inverse $M$-matrices, and to understand their relationship. We start with a short survey of the literature on inverse $Z$-matrices.  

The set of real square matrices of order $n$ is denoted by $\mathbf{M}_n \mathbb{(R)}$. We refer to a matrix each of whose entry is nonzero as a {\it full matrix}. For any matrix $A$, we use $A \geq 0$ to denote the fact that all the entries of $A$ are nonnegative; $A > 0 ~(A < 0)$ will signify that all the entries of $A$ are positive (negative). The determinant of $A$ is denoted by $\det A$. For $\alpha , \beta \subseteq \{1,2,\dotsc,n\}$, whose elements are in the ascending order, we use $A\begin{bmatrix}
    \alpha\\
    \beta
\end{bmatrix}$ to denote the submatrix of $A$ containing the rows and columns indexed by $\alpha$ and $\beta$, respectively. $A[\alpha]$ is referred to as a principal submatrix of $A$. If $\alpha \subsetneq \{1,2, \ldots,n\}$, then $A[\alpha]$ will be called a {\it proper} principal submatrix.  Finally, $A(\alpha)$ denotes the principal submatrix in the rows and columns defined by the complement $\alpha'$ of $\alpha$, in $\{1,2,\dotsc,n\}$. We use $\rho(A)$ to denote the spectral radius of $A,$ which by definition is the maximum of the modulii of the eigenvalues of $A$.

We shall be concerned with the notion of irreducibility of matrices, which we recall, next.

\begin{definition}
For a matrix $A = (a_{i j}) \in \mathbf{M}_n \mathbb{(R)}$, the digraph (directed graph) $\mathcal{D}(A)$ of $A$, has $\{v_1,v_2,\dotsc, v_n\}$ as the vertex set, and whose edge set consists of those ordered pairs $(v_i,v_j),$ if the corresponding entry $a_{i j} \neq 0$.
\end{definition}

Note that the above definition allows the digraph to have loops. A (directed) path in $\mathcal{D}(A)$ from vertex $v_i$ to vertex $v_j$, is a set of distinct vertices $\{v_i, v_{i_1}, \dots, v_{i_r},v_j\}$ such that $(v_i, v_{i_1}), (v_{i_1},v_{i_2}), \dotsc, (v_{i_r}, v_{j})$ are (directed) edges. A matrix $A$ is called {\it irreducible} if its associated digraph $\mathcal{D}(A)$ is strongly connected, i.e., there is at least one directed path between any two vertices in $\mathcal{D}(A)$. 

We shall be interested in the following class of sign pattern matrices. A matrix $A=(a_{ij}) \in \mathbf{M}_n(\mathbb{R})$ is said to be a {\it $Z$-matrix} if $a_{ij} \leq 0$, for all $i \neq j$ (so that all the off-diagonal entries of $A$ are nonpositive). Any such matrix $A$ can be represented as $$A = tI- B, \text{ where } t \in \mathbb{R} \text{ and } B \geq 0.$$ Denote for $r=1,2,\ldots,n$,
$$\rho_r(B):=max \{\rho(\hat{B}): \hat{B}\; \text{is a principal submatrix of}\; B \; \text{of order} \; r\},$$ with the convention that $\rho_0(B):=-\infty$ and $\rho_{n+1}(B):=\infty$.  
The following classification of $Z$-matrices was introduced in \cite{fiedlermark}. Let $L_s(s=0,1,2,\ldots,n)$ denote the class of real $n\times n$ matrices having the form $$A=tI-B,\; \text{where} 
     \; B\geq 0 \;\text{and}\; \rho_s(B)\leq t < \rho_{s+1}(B).$$
A nonsingular matrix $A$ is called an {\it inverse $Z$-matrix}, if $A^{-1}$ is a $Z$-matrix. In order to understand the behaviour of those inverse $Z$-matrices that will be considered here, first we recall the corresponding classes of $Z$-matrices. 

First, we recall the notion of an $M$-matrix. Consider a $Z$-matrix represented as above. If, in addition, one has $t \geq \rho(B)$, then $A$ is called an {\it $M$-matrix}. Such a matrix $A$ is invertible, if $t> \rho(B)$. In that case, it is well known that $A^{-1} \geq 0.$ There is a variety of results that identify when a $Z$-matrix is an $M$-matrix. In the book \cite{Berman} more than fifty characterizations are presented. Two of them are recalled here.

\begin{theorem} \cite[Theorem 2.3]{Berman} \label{m}
Let $A$ be any $Z$-matrix. Then the following are equivalent: 
\begin{itemize}
    \item[1.] $A$ is a nonsingular $M$-matrix; 
    \item[2.] all the principal minors of $A$ are positive; 
    \item[3.]$A^{-1}$ exists and $A^{-1} \geq 0$.
\end{itemize}
\end{theorem}
It follows that the diagonal entries of any nonsingular $M$-matrix are {\it positive}. For the case of irreducible matrices, there is something more to say.

\begin{theorem}\cite[Thorem 2.7]{Berman} 
     Let $A$ be any $Z$-matrix. Then $A$ is a nonsingular irreducible $M$-matrix if and only if $A^{-1} >0$.
\end{theorem}

The other classes of $Z$-matrices are obtained if the number $t$ in the representation $A=tI-B$, for $B \geq 0$, lies to the left of $\rho(B).$ Let us recall one such class, viz., $N$-matrices. Once again, let $A$ be given as above. If $t$ satisfies the inequalities $\rho_{n-1}(B) < t < \rho(B),$ for $n \geq 2$ then $A$ is referred to as an {\it $N$-matrix}. This matrix class was introduced and investigated in \cite{Fan}. The determinant of an $N$-matrix is negative and each proper principal submatrix is an invertible $M$-matrix (so that all its diagonal entries are positive). While an invertible $M$-matrix is inverse nonnegative, an invertible $N$-matrix is inverse negative. In particular, it follows that an $N$-matrix is irreducible. We record this result. 

\begin{theorem}\cite[Corollary 2.8]{Johnson N} \label{j}
Let $A$ be a $Z$-matrix. Then $A$ is a $N$-matrix if and only if $A^{-1}<0$.
\end{theorem}

If equality is allowed in the strict inequality in the definition of an $N$-matrix i.e., in the given representation of $A$, if $t$ satisfies the inequalities $\rho_{n-1}(B) \leq  t < \rho(B)$, then $A$ is referred to as an {\it $N_0$-matrix}. This matrix class was investigated in \cite{Johnson N} and a number of interesting properties were proved. Clearly, every $N$-matrix is an $N_0$-matrix and that the latter matrix class is the topological closure of the former. Interestingly, analogous to the property of any matrix belonging to the class of $N$-matrices, we have the following result: A matrix $A$ is an $N_0$-matrix precisely if $A$ is a $Z$ matrix with the property that all proper principal submatrices are (not necessarily invertible) $M$-matrices, and the determinant of $A$ is negative \cite[Lemma 2.1]{Johnson N}. Here is another result: If $A$ is a $Z$-matrix, then $A$ is an $N_0$-matrix if and only if $A^{-1} \leq 0$ and is irreducible \cite[Theorem 2.7]{Johnson N}. Note that a nonsingular matrix is irreducible if and only if its inverse is irreducible. 

We turn our attention to yet another class which has received considerable attention. A matrix $A\in \mathbf{M}_n(\mathbb{R})$ with $n \geq 3$, is an {\it $F_0$-matrix} if $A=tI-B,$ with $B\geq 0$ and $\rho_{n-2}(B)\leq t < \rho_{n-1}(B)$. An equivalent manner in which this class is defined is given by the following. Let $A\in \mathbf{M}_n(\mathbb{R})$ be a $Z$-matrix. Then $A$ is an $F_0$-matrix if and only if, it satisfies the following conditions:\\
(1) all principal submatrices of order at most $n-2$ are $M$-matrices;\\
(2) at least one principal submatrix of order $n-1$ is an $N_0$-matrix.\\
While there is no nonpositivity or nonnegativity result that is true, in general, for $F_0$-matrices, we have:

\begin{theorem}\cite[Theorem 2.4]{chen}
Let $A$ be a $Z$-matrix. Then $A$ is an $F_0$-matrix if and only if:\\
(1) $det(A) <0;$\\
(2) all principal minors of $A^{-1}$ of order at least two are nonpositive;\\
(3) at least one diagonal entry of $A^{-1}$ is positive.
\end{theorem}
For more details, we refer the reader to \cite{chen} and \cite{Johnson N}. Recalling the definition of $L_s$, it now clear that, for $s=n$, we obtain the $M$-matrix class, the $N_0$-matrix class corresponds to $s=n-1$, while for $s=n-2$, it is the $F_0$-matrix class that one obtains.

Next, we revisit a matrix class, historically important, due to its relevance to inverse $Z$-matrices. Let real numbers $a_1,a_2, \ldots, a_n$ be such that $a_n > a_{n-1} > \ldots > a_1.$ A matrix $A:=(a_{ij})$ is called a {\it matrix of type $D$}, if 
\begin{equation}
 a_{ij} = \begin{cases}
  a_i, & ~{\textit if} ~i \leq j,\\
  a_j, & {\textit{otherwise}}.
  \end{cases} \\
\end{equation}

A matrix of type $D$ of order $n \times n$ will be denoted by $D_n.$ For instance, 
  $$D_4=
  \begin{pmatrix}
   a_1 & a_1 & a_1 & a_1\\
   a_1 & a_2 & a_2 & a_2\\
   a_1 & a_2 & a_3 & a_3\\
   a_1 & a_2 & a_3 & a_4
  \end{pmatrix}.$$
We also have the following recurrence relation:
$$D_{n+1}=
  \begin{pmatrix}
      D_n & b \\
      b^T & a_{n+1}
  \end{pmatrix},$$ where $b=(a_1,a_2, \ldots, a_n)^T.$ Type $D$ matrices were introduced in \cite{markham}, where the first item below was proved. 
  
\begin{theorem}
Let $A\in \mathbf{M}_n(\mathbb{R})$ be a type $D$-matrix.
\begin{enumerate}
    \item If $a_1 >0,$ then $A$ is an inverse $M$-matrix and $A^{-1}$ is tridiagonal \cite[Theorem 2.3]{markham}. 
    \item If $a_n <0,$ then $A$ is an  inverse $N_0$-matrix such that $A^{-1}$ is tridiagonal \cite[Theorem 2.4]{gjohn2}.
    \item  If $a_{n}>0$ and $a_{n-1}<0,$ then $A^{-1}$ is a tridiagonal $F_0$-matrix \cite[Theorem 4.21]{Samir}.   
\end{enumerate}
\end{theorem}

The results above, obtained for three special classes of inverse $Z$-matrices, were extended to the full class $L_s$, which we recall. We set $L_{-1}:=L_n.$

\begin{theorem}\cite[Theorem 3.12]{nabben1}\label{Tridiagonal Z matrix}
    Let $A\in \mathbf{M}_n(\mathbb{R})$ be a matrix of type $D$ with $a_1\neq 0$. Let $s$ denote the number of nonpositive parameters in the sequence $a_n >a_{n-1}>\dotsb>a_1$. Then $A^{-1}$ is a tridiagonal $Z$-matrix and $A^{-1} \in L_{s-1}$. 
 \end{theorem}

In the next result, we obtain specific consequences, presented explicitly, for the three classes: $N$-matrices, $N_0$-matrices and $F_0$-matrices. It appears that these have not been noticed before (see the para after Theorem 3.12, \cite{nabben1}). We skip their proofs. Note that the number of nonpositive parameters in the first and the second items is $n$, while in the third instance, this number is $n-1.$ Observe that the second item below, strengthens the first item above.

\begin{theorem}
    Let $A\in \mathbf{M}_n(\mathbb{R})$ be a type $D$-matrix. 
    \begin{enumerate}
        \item If $a_n <0$, then $A^{-1}$ is a tridiagonal $N$-matrix (and hence a tridiagonal $N_0$-matrix). 
        \item If $a_n=0$, then $A^{-1}$ is a tridiagonal $N_0$-matrix, which is not an $N$-matrix. 
        \item If $a_{n-1}=0$, then $A^{-1}$ is a tridiagonal $F_0$-matrix. 
    \end{enumerate}
\end{theorem}

\subsection{Objective}

In this article, we prove a result each for inverse $M$-matrices and inverse $N$-matrices. While the inverse $M$-matrix result is the same as what is stated in Theorem \ref{motive} to follow, we prove it as a consequence of the main result of this article. The result for inverse $N$-matrices is new. It is pertinent to note that the inverse $M$-matrix problem is to characterize those nonsingular nonnegative matrices that are inverses of $M$-matrices. While a number of matrix classes have been shown to be inverse $M$-matrices, the problem in its full generality remains open. We refer the reader to the recent monograph \cite{jts} for a comprehensive treatment.

Consider the following two results.
\begin{theorem}\cite[Theorem 4.3]{McDonald unipathic}\label{motive}
    Let $A =(a_{ij}) \in \mathbf{M}_{n}(\mathbb{R})$. Then the following statements are equivalent: \begin{enumerate}
        \item $A$ is an inverse $M$-matrix such that $\mathcal{D}(A^{-1})$ is the simple $n$-cycle $v_1 \rightarrow v_2 \rightarrow \dots \rightarrow v_n \rightarrow v_1$ (with loops).
        \item $A >0$ and the following hold: 
        \begin{enumerate}
            \item $a_{11}a_{22} \dotsb a_{nn} > a_{12}a_{23} \dotsb a_{(n-1)n}a_{n1}$,
            \item $a_{ij}=\dfrac{a_{ik}a_{kj}}{a_{kk}},$ for all distinct vertices $v_i,v_j,v_k$ such that $v_k$ lies on the path from $v_i$ to $v_j$.
        \end{enumerate}
    \end{enumerate}
\end{theorem}

Set $Z:=(e^n,e^1,e^2, \ldots,e^{n-1})$, where $e^k$ denotes the $k^{th}$ standard basis vector of $\mathbb{R}^n$.

\begin{theorem}\cite[Corollary 4.5]{McDonald unipathic}\label{cycle permutation}
    Let $Z$ be the matrix defined as above. Let $\alpha_1, \alpha_2, \dotsc, \alpha_n$ be nonnegative numbers and 
    $$A:=p(Z)=\alpha_1I +\alpha_2Z+ \dotsb +\alpha_nZ^{n-1}.$$ Then the following are equivalent:
    \begin{enumerate}
        \item $A^{-1}$ is an $M$-matrix with $\mathcal{D}(A^{-1})$ being the simple $n$-cycle $v_1 \rightarrow v_2 \rightarrow \dots \rightarrow v_n \rightarrow v_1$, with loops.
        \item $A>0$ satisfies 
        \begin{enumerate}
            \item $\alpha_1 >\alpha_2 >0$,
            \item $\alpha_r =(\alpha_2)^{r-1}/(\alpha_1)^{r-2},$ for $r=3, \dotsc,n$.
        \end{enumerate}
    \end{enumerate}
\end{theorem}

The motivation for the article comes from the question as to what lies in the background of Theorem \ref{motive}, which when applied to matrices satisfying the second condition, gives rise to the special class of inverse $M$-matrices, described in the first item. Specifically, inspired by the first item above, we introduce a matrix class, without requiring it to be an $M$-matrix. We call such a matrix a {\it bi-diagonal south-west matrix}. Simultaneously, another new class is identified, taking cue from the second item, giving rise to what we refer to as an {\it inverse cyclic matrix.}

A relationship is established between these two matrix classes in Theorem \ref{Cycle matrix}. This result, while enabling us to recover Theorem \ref{motive} as a particular case, further allows us to obtain a new result on inverse $N$-matrices, which is presented in Theorem \ref{bdsw Z matrix}. An analogue of Theorem \ref{cycle permutation} is obtained in Theorem \ref{polyn}, for inverse $N$-matrices.

\section{Inverse cyclic matrices and bi-diagonal south-west matrices}

We propose two matrix classes and study their relationship. The first notion is that of the {\it inverse cyclic property} of a matrix.

\begin{definition}
    Let $A=(a_{ij})  \in \mathbf{M}_n(\mathbb{R})$ be such that $a_{ii}\neq 0, 1 \leq i \leq n.$ Then $A$ is said to have the inverse cyclic property (or is called an inverse cyclic matrix), if the entries of $A$ satisfy the following conditions:
    \begin{equation}\label{cyclic property}
        a_{ij}=\begin{cases}
            \frac{a_{ik}a_{kj}}{a_{kk}}, & \text{ for } i <k <j,\\
            \frac{a_{in}a_{nj}}{a_{nn}}, & \text{ for } j<i \neq n,\\
            \frac{a_{n1}a_{1j}}{a_{11}}, & \text{ for } j<i=n.
        \end{cases}
    \end{equation}
\end{definition}

The requirement on the entries above can also be written equivalently, as 
 \begin{equation}\label{equivalent cyclic property}
        a_{ij}=\begin{cases}
            \frac{a_{i(i+1)}a_{(i+1)(i+2)} \dotsb a_{(j-1)j}}{a_{(i+1)(i+1)}a_{(i+2)(i+2)} \dotsb a_{(j-1)(j-1)}}, & \text{ for } i+1<j,\\
            \frac{a_{i(i+1)}a_{(i+1)(i+2)} \dotsb a_{(n-1)n} a_{n1} a_{12} \dotsb a_{(j-1)j}}{a_{(i+1)(i+1)} \dotsb a_{nn}a_{11} \dotsb a_{(j-1)(j-1)}}, & \text{ for } j<i \neq n,\\
            \frac{a_{n1}a_{12}a_{23} \dotsb a_{(j-2)(j-1)}a_{(j-1)j}}{a_{11}a_{22} \dotsb a_{(j-1)(j-1)}}, & \text{ for } j<i=n.
        \end{cases}
    \end{equation}

For an inverse cyclic matrix, it follows that all the elements in the upper triangular part are determined by the elements on the main diagonal and the super diagonal, while the elements in the lower triangular part are determined by those on the main diagonal, the super diagonal and the entry lying in the intersection of the first column and the last row of the matrix. 

\begin{example}\label{nonsingular cycle matrix}
It may be verified that the matrix $A=\begin{pmatrix}
        ~~2 & -2 &-4 &~~0\\
        ~~0 &~~1&~~2 &~~0\\
        ~~0&~~0&-2&~~0\\
        ~~2 & -2 &-4 &~~1
    \end{pmatrix}$ 
has the inverse cyclic property.
\end{example}

For $A=(a_{ij})  \in \mathbf{M}_n(\mathbb{R})$, set $$d:=a_{11}a_{22}\dotsb a_{(n-1)(n-1)}a_{nn},$$
i.e. the product of the diagonal entries of $A$, 
$$c:=a_{12}a_{23} \dotsb a_{(n-1)n}a_{n1},$$
which we may refer to as the {\it cyclic} product. Let us denote by $x^k$, the $k^{th}$ column of the matrix $X=(x_{ij})$.

\begin{theorem}\label{det of a cycle matrix}
    Let $A \in \mathbf{M}_n(\mathbb{R})$ be an inverse cyclic matrix. Then $$\det A =\frac{(d - c)^{n-1}}{d^{n-2}}.$$
\end{theorem}
\begin{proof}
First, note that an inverse cyclic matrix is one where all the diagonal entries are nonzero. Thus, $d \neq 0.$ Let $C$ be the matrix obtained from $A$, after employing the following column operations:
\begin{equation}\label{REF operation}
    c^k=a^k - \dfrac{a_{(k-1)k}}{a_{(k-1)(k-1)}} a^{k-1}, ~2 \leq k \leq n.
    \end{equation}
      It is easy to see that $C$ is a lower triangular matrix with diagonal entries $$c_{ii}=\begin{cases}
        \hspace{2.7cm} a_{11}, & \text{ for }i=1,\\
        a_{ii} - \frac{ a_{12}a_{23} \dotsb a_{(n-1)n}a_{n1}}{a_{11}a_{22} \dotsb a_{(i-1)(i-1)} a_{(i+1)(i+1)} \dotsb a_{nn}}, & \text{ for } i \neq 1.
    \end{cases}$$
Thus, we have $$c_{ii}=\begin{cases}
        \hspace{.3cm} a_{11}, & \text{ for }i=1,\\
        \frac{d-c}{d}a_{ii}, & \text{ for } i \neq 1.
    \end{cases}$$
Then,
\begin{eqnarray*}
        \det C=\prod\limits_{\substack{i=1}}^{n} c_{ii} &=&  \Big(\frac{d-c}{d}\Big)^{n-1} {a_{11}a_{22}a_{33} \dotsb a_{nn}}\\
        &=& \dfrac{(d-c)^{n-1}}{d^{n-2}}.
    \end{eqnarray*}
\end{proof}

Next, we introduce the second new class of matrices called {\it bi-diagonal south-west} matrices.

\begin{definition}
Let $A \in \mathbf{M}_n(\mathbb{R})$. Suppose that all the entries of $A$ that lie on the main diagonal, the super diagonal and the entry in the south-west corner, are nonzero, while all the other entries are zero. Then $A$ is called a bi-diagonal south-west matrix, or bdsw, for short.
\end{definition}
As an illustration, a pattern of the bdsw matrix (for $n=4$), is given by 
$$\begin{pmatrix}
    * & * & 0&0\\
    0 & * & * &0\\
    0&0&* &*\\
    * &0 &0 &*
\end{pmatrix},$$ where the symbol $*$ stands for a nonzero entry.

\begin{remark}\label{structure of proper principal submatrix of bdsw matrix}
Let $A$ be a bdsw matrix. Then every proper principal submatrix of $A$ is either an upper triangular matrix or a block lower triangular matrix. Moreover, every proper principal submatrix of a bdsw matrix is nonsingular. 
\end{remark}

The motivation for considering the class of bdsw matrices arises from the structure of the digraph associated with it. The digraph corresponding to a bdsw matrix (of order $n \times n$) is the simple directed $n$-cycle $v_1 \rightarrow v_2 \rightarrow \dots \rightarrow v_n \rightarrow v_1$, with loops. The digraph for a bdsw matrix of order $4$ is  given below.

\begin{figure}[H]
    \centering
   \begin{tikzpicture}
       % Drawing the vertices
       \foreach \i/\j/\k/\l in {0/0/1/135, 0/2/2/45, 2/2/3/315, 2/0/4/225}{
           \draw (\i,\j) node[circle, fill=black, inner sep=2pt, label=\l:{$v_\k$}](\k){};
       }
       
      \draw[thick] (0,0) -- (2,0) node[midway] {\tikz \draw[<-, thick, scale=3] (0,0) -- (0.01,0);};
\draw[thick] (0,2) -- (2,2) node[midway] {\tikz \draw[->, thick, scale=3] (0,0) -- (0.01,0);};

    \draw[thick, black, decoration={markings, mark=at position 0.5 with {\arrow[scale=1]{>}}}, postaction={decorate}] 
      (0, 0) .. controls (0, 1) .. (0, 2);
      \draw[thick, black, decoration={markings, mark=at position 0.5 with {\arrow[scale=1]{<}}}, postaction={decorate}] 
      (2, 0) .. controls (2, 1) .. (2, 2);
      
       % Self-loops on vertices with arrows
       \draw[thick, ->] (1) to[in=180, out=270, looseness=15] (1);
       \draw[thick, ->] (2) to[in=90, out=180, looseness=15] (2);
       \draw[thick, ->] (3) to[in=90, out=360, looseness=15] (3);
       \draw[thick, ->] (4) to[in=0, out=270, looseness=15] (4);
   \end{tikzpicture}
    \caption{Digraph of a $4 \times 4$ bdsw matrix.}
    \label{fig:Digraph bdsw matrix}
\end{figure}
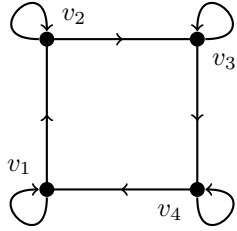

Recall that {\it unipathic} digraphs are those digraphs in which there is at most one (directed) path between any two vertices. However, matrices $A$ possessing the bdsw structure are examples of irreducible unipathic matrices. This means that, in $\mathcal{D}(A)$, there is exactly {\it one} path between any two vertices.

We will make use of an auxiliary result, which presents a graph theoretic formula for the inverse of a nonsingular matrix. Let us recall some terminology, in this context. Let $p(v_i \to v_j)$ denote any (directed) path from $v_i$ to $v_j$ and $l(p),$ the length of the path $p$. The set of vertices of $\mathcal{D}(A)$ not belonging to the path $p$ will be denoted by $V(p)$, and by $A[p(v_i \to v_j)]$ we signify the product of those elements of $A$, that lie along the given path $p$ of $\mathcal{D}(A)$, from vertex $v_i$ to vertex $v_j$. We set, $A[p(v_i \to v_j)] = 0$ if there is no path from vertex $v_i$ to vertex $v_j$. 

\begin{theorem}\cite[Corollary 9.1]{Maybee}\label{entries of $A^{-1}$}
Let $A=(a_{ij}) \in \mathbf{M}_n \mathbb{(R)}$ be nonsingular. Set $A^{-1} = (\tilde{a}_{i j})$. Then, we have $$ \tilde{a}_{i i} = \det A(i)/ \det A$$ and 
\begin{equation} \label{inverse entry}\tilde{a}_{i j}= \frac{1}{\det A} \sum_{p(v_i \to v_j)} (-1)^{l(p)} A[p(v_i \to v_j)] \det A[V(p)], \;\;\; \text{ for } i \neq j. 
\end{equation} 
\end{theorem}

We will make use of the following particular case.

\begin{corollary}
Let $A$ be a nonsingular unipathic matrix. Then, the off-diagonal entries of $A^{-1}$ are given by \begin{equation}\label{inverse entries}
    \tilde{a}_{ij}= (-1)^{l(p)} A[p(v_i \to v_j)] \det A[V(p)]/ \det A, \text{ for } i \neq j,
\end{equation} where $p$ is the unique path from vertex $v_i$ to vertex $v_j$, if it exists.
\end{corollary}

\begin{lemma}\label{inverse of bdsw is full matrix}
    The inverse of a nonsingular bdsw matrix is a full matrix.
\end{lemma}
\begin{proof}
Let $A$ be a nonsingular bdsw matrix. Then $A$ is a unipathic matrix, so that the entries of $A^{-1}=(\tilde{a}_{ij})$ satisfy (\ref{inverse entries}). Moreover, $A$ is irreducible and so $A[p(v_i \to v_j)] \neq 0$, for all $i,j$. Also, from Remark \ref{structure of proper principal submatrix of bdsw matrix}, $\det A[V(p)] \neq 0$, for any path $p$. In particular, all the diagonal entries of $A^{-1}$ are nonzero and hence $\tilde{a}_{ij} \neq 0, \text{ for all } i,j.$ Thus, $A^{-1}$ is a full matrix. 
\end{proof}

\begin{remark} 
The converse of Lemma \ref{inverse of bdsw is full matrix} is false. For example, the matrix $$A=\begin{pmatrix}
    1 & 2&2\\
    1 &1&1\\
    2&2&1
\end{pmatrix},$$ is a nonsingular full matrix, But its inverse $$A^{-1}=\begin{pmatrix}
    -1 & ~~2 &~~0\\
    ~~1 &-3 &~~1\\
    ~~0 & ~~2 &-1
\end{pmatrix},$$ is not a bdsw matrix.
\end{remark}

Now, we prove the main result of this article. This brings about the relationship between {\it nonsingular} matrices possessing the inverse cyclic property and bdsw matrices. This result also justifies the nomenclature “inverse cyclic matrix". We shall make use of the following auxiliary result, which presents a formula for the determinant of a submatrix of a nonsingular matrix and a related submatrix of its inverse. 

\begin{lemma}(\cite{Karlin}, pp. 5) \label{inv det} 
Let $A \in \mathbf{M}_{n}(\mathbb{R})$ be nonsingular and $B:=A^{-1}$. Let $\alpha =\{i_{1}, i_{2}, \dotsc, i_{p}\}$ and $\beta =\{j_{1}, j_{2}, \dotsc, j_{p} \}$ with $1 \leq i_{1} < i_{2} < \dots < i_{p} \leq n$ and $1 \leq j_{1} < j_{2} < \dots < j_{p} \leq n$. Set $\gamma_p:=  
(-1)^{\sum\limits_{m=1}^{p} (i_{m}+j_{m})}.$ 
We then have:   
\begin{equation}\label{det inverse} 
    \det B \begin{bmatrix}
    \alpha \\ 
     \beta
\end{bmatrix} = \frac{\gamma_p}{\det A } \det A\begin{bmatrix}
    \beta' \\ 
     \alpha'
\end{bmatrix}.
\end{equation}
\end{lemma}

\begin{theorem}\label{Cycle matrix}
    Suppose that $A \in \mathbf{M}_n(\mathbb{R})$ is nonsingular. Then $A$ is a full matrix and has the inverse cyclic property, if and only if $A^{-1}$ is a bdsw matrix.
\end{theorem}
\begin{proof}
Let us assume  that $A$ is a full matrix and has the inverse cyclic property. Since $A$ is nonsingular, by Theorem \ref{det of a cycle matrix}, $d-c \neq 0$. Define $B=(b_{ij})$, where
\begin{equation}\label{inverse entries of cyclic matrix}
    b_{ij}=\dfrac{1}{d-c}\begin{cases}

    -a_{ij}\Big(\prod\limits_{\substack{k=1 \\ k \neq i,j}}^{n} a_{kk}\Big), & \text{ if } j=i+1,\\ 

 -a_{ij}\Big(\prod\limits_{\substack{k=2 }}^{n-1} a_{kk}\Big), & \text{ if } i=n, j=1,\\ 
        
     \hspace{1.2cm}   \prod\limits_{\substack{k=1 \\ k \neq i}}^{n} a_{kk}, & \text{ if } i=j,\\
     \hspace{1.4cm} 0, &   \text{ otherwise}.
    \end{cases}
\end{equation}

We claim that $B=A^{-1}$. Since $A$ is a full matrix, it is clear that $B$ is a bdsw matrix. Now, for all $i \in \{1,2,\dotsc, n\}$, \begin{equation}\label{Product intries of AB}
    (AB)_{ij}=\begin{cases}
    a_{i1}b_{11} + a_{in}b_{n1}, & \text{ if } j=1,\\

    a_{i(j-1)}b_{(j-1)j} + a_{ij}b_{jj}, & \text{ if } j \neq 1.
\end{cases}
\end{equation}
We have, for $2 \leq j \leq n$, 
$$\frac{b_{(j-1)j}}{b_{jj}} =-\frac{a_{(j-1)j}}{a_{(j-1)(j-1)}}$$ and $$\frac{b_{n1}}{b_{11}}=-\frac{a_{n1}}{a_{nn}}.$$

Incorporating these in (\ref{Product intries of AB}), we obtain \begin{equation}\label{entries of product AB}
    (AB)_{ij}=\begin{cases}
    \big(a_{i1}-  \frac{a_{in}a_{n1}}{a_{nn}}\big)b_{11}, & \text{ for } j=1,\\
    \big( a_{ij} - \frac{a_{i(j-1)}a_{(j-1)j}}{a_{(j-1)(j-1)}} \big) b_{jj},  & \text{ for } j \neq 1.
\end{cases}
\end{equation}
\textbf{Case 1:} $i = j$.\\
First, let $i=j=1$. Then, using (\ref{equivalent cyclic property}), we have
\begin{eqnarray*}
    (AB)_{11}& = & \big(a_{11}-  \frac{a_{1n}a_{n1}}{a_{nn}}\big)b_{11} \\
    &=&  \big(a_{11}-  \frac{a_{12}a_{23} \dotsb a_{(n-1)n}a_{n1}}{a_{22} a_{33} \dotsb a_{(n-1)(n-1)}a_{nn}}\big) b_{11}\\
    &=& \frac{a_{11}a_{22} \dotsb a_{nn} -a_{12}a_{23} \dotsb a_{(n-1)n}a_{n1} }{a_{22}a_{33} \dotsb a_{nn}} b_{11}=1.
\end{eqnarray*}
Next, let $i =j \neq 1$. Then, using (\ref{equivalent cyclic property}) again, we have
\begin{eqnarray*}
    (AB)_{ii} & = & \big(a_{ii} - \frac{a_{i(i-1)}a_{(i-1)i}}{a_{(i-1)(i-1)}} \big) b_{ii} \\
    & = & \big(a_{ii} - \frac{a_{i(i+1)}a_{(i+1)(i+2)} \dotsb a_{(n-1)n} a_{n1} a_{12} \dotsb a_{(i-2)(i-1)}a_{(i-1)i}}{a_{(i+1)(i+1)} \dotsb a_{nn} a_{11} \dotsb a_{(i-2)(i-2)}a_{(i-1)(i-1)}} \big) b_{ii}\\
    & = &  \frac{a_{11}a_{22} \dotsb a_{nn} -a_{12}a_{23} \dotsb a_{(n-1)n}a_{n1} }{a_{11}a_{22} \dotsb a_{(i-1)(i-1)} a_{(i+1)(i+1)} \dotsb a_{nn}} b_{ii}=1.
\end{eqnarray*}
\textbf{Case 2:} $i \neq j$.\\
First, let $j=1$. We have $a_{i1}=\dfrac{a_{in}a_{n1}}{a_{nn}}$, so that $(AB)_{i1}=0$. \\
Next, we consider $j \neq 1$. We have two cases. First, let $i <j$. Then
$$ \frac{a_{i(j-1)}a_{(j-1)j}}{a_{(j-1)(j-1)}}=a_{ij} .$$ Also for $j<i<n$, 
\begin{eqnarray*}
    \frac{a_{i(j-1)}a_{(j-1)j}}{a_{(j-1)(j-1)}}& =& \frac{a_{i(i+1)}a_{(i+1)(i+2)} \dotsb a_{n1}a_{12} \dotsb a_{(j-2)(j-1)}a_{(j-1)j}}{a_{(i+1)(i+1)}a_{(i+2)(i+2)} \dotsb a_{nn}a_{11} \dotsb a_{(j-2)(j-2)}a_{(j-1)(j-1)}} \\
    &=& a_{ij},
\end{eqnarray*}
and for $j <i=n$,
\begin{eqnarray*}
   \frac{a_{n(j-1)}a_{(j-1)j}}{a_{(j-1)(j-1)}}& =& \frac{a_{n1}a_{12} \dotsb a_{(j-2)(j-1)}a_{(j-1)j}}{a_{11} a_{22}\dotsb a_{(j-2)(j-2)}a_{(j-1)(j-1)}}\\
    &=& a_{nj}.
\end{eqnarray*}
Then from (\ref{entries of product AB}), it follows that $(AB)_{ij}=0,$ in this case, too. This completes the proof that $AB=I$. \\
Conversely, suppose that $B=A^{-1}$ is a bdsw matrix. Then from Lemma \ref{inverse of bdsw is full matrix}, $A$ is a full matrix and so, in particular all the diagonal entries of $A$ are nonzero. Now, using Lemma \ref{inv det}, we will show that $A$ has the inverse cyclic property. This is equivalent to proving that the following are true:
\begin{equation}\label{upper part}
    \det A\begin{pmatrix}
        i &k\\
        k &j
    \end{pmatrix}=0, \text{ for } i<k<j,
\end{equation} 
\begin{equation}\label{lower part}
    \det A\begin{pmatrix}
        i &n\\
        j &n
    \end{pmatrix}=0, \text{ for } j<i \neq n
\end{equation}
and
\begin{equation}\label{last row}
    \det A\begin{pmatrix}
        1&n\\
        1 &j
    \end{pmatrix}=0, \text{ for } j<i=n.
\end{equation}
In view of Lemma \ref{inv det}, it suffices to show that the following hold, for the matrix $B$:
\begin{equation}\label{equivalent upper part}
 \det B\begin{pmatrix}
        1,2, \dots, k-1,k+1, \dots , j-1 ,j+1, \dots, n\\
        1,2, \dots, i-1,i+1, \dots, k-1,k+1 \dots, n
    \end{pmatrix}=0, \text{ for } i<k<j,
\end{equation}

\begin{equation}\label{equivalent lower part}
    \det B\begin{pmatrix}
        1,2, \dots, j-1,j+1, \dots , n-1 \\
        1,2, \dots, i-1,i+1, \dots, n-1
    \end{pmatrix}=0, \text{ for } j<i \neq n,
\end{equation}

\begin{equation}\label{equivalent last row}
     \det B\begin{pmatrix}
        2,3, \dots, j-1,j+1, \dots , n\\
        2,3, \dots, \dots ,  n-1
    \end{pmatrix}=0, \text{ for }j<i=n.
\end{equation}

Consider the case $i <k<j$. Let $E_1$ be the submatrix of $B$ of order $n-2$, corresponding to the left hand side of (\ref{equivalent upper part}). Let $C_{1}$ be the submatrix of $E_1$ obtained by taking the {\it first} $k-1$ rows and {\it all} the $n-2$ columns. Consider the submatrix $F_{1}$ of $C_{1}$, obtained by taking all the $k-1$ rows and the {\it first} $k-2$ columns of $C_{1}$. The rank of $F_1$ does not exceed $k-2.$ Since $B$ is a bdsw matrix, and since the entries along all the other columns of $C_{1}$, not being in the submatrix $F_{1}$, are zero, 
it follows that the $k-1$ rows of $C_{1}$ are linearly dependent. Thus, $\det E_1=0,$ showing that (\ref{equivalent upper part}) holds.

Next, let $j<i \neq n$. Suppose that $E_2$ is the submatrix of $B$ of order $n-2$, corresponding to the left hand side of (\ref{equivalent lower part}). Let $C_{2}$ be the submatrix of $E_2$ obtained by taking the {\it last} $n-i$ rows and {\it all} the $n-2$ columns. Consider the submatrix $F_{2}$ of $C_{2}$, obtained by taking {\it all} the $n-i$ rows and the {\it last} $n-i-1$ columns of $C_{2}$. The rank of $F_2$ does not exceed $n-i-1.$ Moreover, since the entries along all the other columns of $C_{2}$, not being in the submatrix $F_{2}$, are zero (since $B$ is a bdsw matrix), it follows that the $n-i$ rows of $C_{2}$ are linearly dependent. Thus, $\det E_2=0,$ showing that (\ref{equivalent lower part}) holds.

Finally, let $j<i=n$. Denote by $E_3$, the submatrix of $B$ of order $n-2$, corresponding to the left hand side of (\ref{equivalent last row}). Again, since $B$ is a bdsw matrix, all the entries in its {\it last} row except the first and the last entries, are zero. Thus, all the entries in the {\it last row} of $E_3$ are zero, so that $\det E_3=0$. Thus (\ref{equivalent last row}) holds.\\
We have shown that $A$ has the inverse cyclic property, completing the proof.
\end{proof}

\begin{example}\label{example for the main theorem}
     Let $A=\begin{pmatrix}
        ~~1 &-1&-1\\
        -2&~~1 &~~1\\
        ~~2 &-2&-1
    \end{pmatrix}$.
    Then $A$ satisfies the inverse cyclic property. Also, 
 $$A^{-1}=\begin{pmatrix}
        -1&-1&~~0\\
        ~~0&-1&-1\\
        -2&~~0&~~1
    \end{pmatrix},$$ is a bdsw matrix. Consider the matrix $C,$ obtained by changing $A$ in precisely one entry ($a_{32}$), given by $$C=\begin{pmatrix}
         ~~1 &-1&-1\\
        -2&~~1 &~~1\\
        ~~2 &~~2&-1
    \end{pmatrix}.$$ Since $c_{32} \neq \dfrac{c_{31}c_{12}}{c_{11}}$, the matrix $C$ does not possess the inverse cyclic property. Here, 
 $$C^{-1}=\frac{1}{3}\begin{pmatrix}
       -3&-3&~~0\\
       ~~0 &~~1 &~~1\\
       -6&-4&-1
    \end{pmatrix}$$
is not a bdsw matrix. 
\end{example}

\begin{remark}
It is easy to verify that any principal submatrix of an inverse cyclic matrix inherits the property. Also, if $A$ is a nonsingular matrix having the inverse cyclic property and $B=A^{-1},$ then for all $ 1\leq i \leq n-1,$ $$a_{ii}b_{ii}=a_{(i+1)(i+1)}b_{(i+1)(i+1)}=\dfrac{d}{d-c}.$$ 
\end{remark}

\section{Inverse $Z$-matrices with the bdsw structure}
In this section, we present applications of Theorem \ref{Cycle matrix} to the two classes of nonsingular $Z$-matrices under study here, viz., inverse $M$-matrices and inverse $N$-matrices.

\begin{theorem}\label{bdsw Z matrix}
For $A \in \mathbf{M}_n(\mathbb{R}),$ we have the following:
\begin{enumerate}
    \item $A$ is an inverse $M$-matrix such that $A^{-1}$ possesses the bdsw structure if and only if $A>0$, $d-c>0$ and $A$ has the inverse cyclic property.
    \item Let $A$ be an even order matrix. Then $A$ is an inverse $N$-matrix, $A^{-1}$ has the bdsw pattern if and only if $A<0$, has the inverse cyclic property with $d-c <0$.
    \item Let $A$ be an odd order matrix. Then $A$ is an inverse $N$-matrix such that $A^{-1}$ is endowed with the bdsw structure if and only if $A<0$ has the inverse cyclic property and $d-c >0$.
\end{enumerate}
\end{theorem}
\begin{proof} 
(1) Let $A$ be an inverse $M$-matrix such that $A^{-1}$ has the bdsw structure. Since $A^{-1}$ is an $M$-matrix, then all the diagonal entries of $A$ are nonzero and $A=(A^{-1})^{-1} \geq 0$. Moreover, the bdsw structure of $A^{-1}$, implies that $A$ is a full matrix i.e., $A>0$ and $A$ has the inverse cyclic property (by Theorem \ref{Cycle matrix}). Also, the entries of $A^{-1}=:B=(b_{ij})$ are given by formula (\ref{inverse entries of cyclic matrix}).

Since $B$ is an $M$-matrix, i.e., $b_{ij} \leq 0$, for $i \neq j,$ and the entries of $A$ are positive, it follows from (\ref{inverse entries of cyclic matrix}) that $d-c >0$.

Conversely, assume that $A>0$ has the inverse cyclic property and $d-c >0$. Then $A$ is nonsingular and by Theorem \ref{Cycle matrix}, $A^{-1}$ is a bdsw matrix. To show that $A^{-1}$ also an $M$-matrix, it is enough to show that $A^{-1}$ is a $Z$-matrix, since $A>0$. However, this is easy to see, by the formula for $A^{-1}$ given in (\ref{inverse entries of cyclic matrix}).

Items (2) and (3) follow in an entirely similar manner, with the additional observation that the sign of the products appearing in the inverse formula depend on the odd-even parity of the number of factors.
\end{proof}

The following examples illustrate the above Theorem.
\begin{example}
    Consider the matrix $$A=\begin{pmatrix}
        4&4&8&4&4\\
        1&2&4&2&2\\
        1&1&4&2&2\\
        2&2&4&4&4\\
        2&2&4&2&4
    \end{pmatrix}.$$ It may be verified that $A$ has the inverse cyclic property and $d-c>0$. Also $$A^{-1}=\frac{1}{4}\begin{pmatrix}
         ~~2& -4&  ~~0&  ~~0&  ~~0\\
 ~~0&  ~~4& -4&  ~~0&  ~~0\\
 ~~0&  ~~0&  ~~2& -1&  ~~0\\
 ~~0&  ~~0&  ~~0&  ~~2& -2\\
-1&  ~~0&  ~~0&  ~~0&  ~~2
    \end{pmatrix},$$ a bdsw $M$-matrix.
\end{example}

\begin{example}
    Consider the (even order) matrix $$A=\begin{pmatrix}
        -2&-2&-4&-8\\
        -4&-1&-2&-4\\
        -2&-2&-1&-2\\
        -2&-2&-4&-2
    \end{pmatrix}.$$ Then, $A$ has the inverse cyclic property and $d-c<0$. Its inverse is given by 
    $$A^{-1}=\frac{1}{6}\begin{pmatrix}
        ~~1 & -2 &  ~~0 &  ~~0\\
~~0 &  ~~2 & -4&  ~~0\\
~~0 &  ~~0&  ~~2& -2\\
-1&  ~~0&  ~~0&  ~~1
    \end{pmatrix},$$ 
    which is a bdsw $N$-matrix.
\end{example}

\begin{example}
     The odd order matrix $$A=\begin{pmatrix}
        -2&-2&-4&-8&-16\\
        -8&-1&-2&-4&-8\\
        -4&-4&-1&-2&-4\\
        -2&-2&-4&-1&-2\\
        -2&-2&-4&-8&-2
    \end{pmatrix},$$ is an inverse cyclic matrix, with $d-c>0$. Here, $$A^{-1}=\frac{1}{14}\begin{pmatrix}
         ~~1& -2&  ~~0&  ~~0& ~~0\\
~~ 0&  ~~2& -4&  ~~0&  ~~0\\
~~0&  ~~0&  ~~2& -4&  ~~0\\
~~0&  ~~0& ~~0&  ~~2& -2\\
-1&  ~~0&  ~~0&  ~~0&  ~~1
    \end{pmatrix},$$ is a bdsw $N$-matrix.

\end{example}

\begin{remark}
The following examples illustrate that the odd-even parity in the order of the matrix is indispensable in Theorem \ref{bdsw Z matrix}. The matrix $$A=\begin{pmatrix}
        -2 &-2 &-2 &-2\\
        -1&-2&-2&-2\\
        -1&-1&-2&-2\\
        -1&-1&-1&-2
    \end{pmatrix}$$ of even order has the inverse cyclic property with $d-c >0$. While $$A^{-1}=\frac{1}{2}\begin{pmatrix}
        -2&  ~~2&  ~~0 &  ~~0\\
~~0& -2&  ~~2&  ~~0\\
~~0&  ~~0& -2&  ~~2\\
~~1&  ~~0&  ~~0& -2
    \end{pmatrix}$$ is a bdsw matrix, it is not an $N$-matrix. On the other hand, if $$A=\begin{pmatrix}
        -2&-2&-2&-2&-2\\
        -1&-2&-2&-2&-2\\
        -1&-1&-2&-2&-2\\
        -1&-1&-1&-2&-2\\
        -1&-1&-1&-1&-2
    \end{pmatrix},$$ then $A$ is of odd order, $d-c <0$ and has the inverse cyclic property. Here $$A^{-1}=\frac{1}{2}\begin{pmatrix}
        -2&  ~~2&  ~~0&  ~~0&  ~~0\\
~~0& -2&  ~~2&  ~~0&  ~~0\\
~~0&  ~~0& -2&  ~~2&  ~~0\\
~~0&  ~~0&  ~~0& -2&  ~~2\\
~~1&  ~~0&  ~~0&  ~~0& -2
  \end{pmatrix},$$ is not even a $Z$-matrix.   
\end{remark}

We conclude this article with a result for inverse $N$-matrices, analogous to Theorem \ref{cycle permutation}.

\begin{theorem}\label{polyn}
    Let $Z$ be defined as in Theorem \ref{cycle permutation}. Let $\alpha_1, \alpha_2, \dotsc, \alpha_n$ be nonpositive numbers and $$A:=p(Z)=\alpha_1I +\alpha_2Z+ \dotsb +\alpha_nZ^{n-1}.$$ Then the following are equivalent:
    \begin{enumerate}
        \item $A^{-1}$ is an $N$-matrix with $\mathcal{D}(A^{-1})$ being the simple $n$-cycle $v_1 \rightarrow v_2 \rightarrow \dots \rightarrow v_n \rightarrow v_1$, with loops.
        \item $A<0$ satisfies 
        \begin{enumerate}
            \item $\alpha_2 <\alpha_1 <0$,
            \item $\alpha_r =(\alpha_2)^{r-1}/(\alpha_1)^{r-2},$ for $r=3, \dotsc,n$.
        \end{enumerate}
    \end{enumerate}
\end{theorem}
\begin{proof}
    It may be verified that $$A=\begin{pmatrix}
        \alpha_1 & \alpha_2 & \dots &\alpha_{n-1} &\alpha_n\\
        \alpha_n & \alpha_1 & \alpha_2 & \dots & \alpha_{n-1}\\
        \vdots & \alpha_n & \alpha_1 & \ddots & \vdots\\
        \alpha_3 & \vdots & \ddots & \alpha_1 & \alpha_2\\
        \alpha_2 & \alpha_3 & \dots & \alpha_n & \alpha_1
    \end{pmatrix}.$$
Suppose that (1) holds. Then $A^{-1}$ is an $N$-matrix with the bdsw structure. By Theorem \ref{bdsw Z matrix}, we then have $A<0$ and $A$ possesses the inverse cyclic property. Further, $d-c <0$, if $n$ even and $d-c >0$, for $n$ odd. Moreover, the inverse cyclic property of $A$ yields (b). Finally, the sign of $d-c$ for the odd-even parity of the order of the matrix, yields the inequality $\alpha_2 < \alpha_1 <0$.

Conversely, suppose that $A<0$ satisfies the conditions $(a)$ and $(b)$. Then, $A$ is a full matrix with the inverse cyclic property. The inequalities $\alpha_2 <\alpha_1 <0$ imply that $d-c <0$, if $n$ even and $d-c >0$, for $n$ odd. By Theorem \ref{bdsw Z matrix} again, $A^{-1}$ is an $N$-matrix, with the bdsw structure, so that $\mathcal{D}(A^{-1})$ is a simple cycle $v_1 \rightarrow v_2 \rightarrow \dots \rightarrow v_n \rightarrow v_1$, with loops, completing the proof.
\end{proof}

\section{Concluding remarks}
We propose two new classes of nonsingular matrices and establish a relationship between them. We obtain consequences for two subclasses of inverse $Z$-matrices. It is natural to ask what versions of
these results apply to other inverse Z-matrices, in general and the classes of $N_0$-matrices and $F_0$-matrices, in particular. This is a problem for a future study.\\

\bibliographystyle{amsalpha}

\end{document}